\documentclass[11pt]{article}
\usepackage{amssymb,amsmath,amsthm,latexsym,graphicx}

\newtheorem{theorem}{Theorem}[section]
\newtheorem{lemma}[theorem]{Lemma}

\newtheorem{corollary}[theorem]{Corollary}

\newtheorem{claim}[theorem]{Claim}

\newtheorem{remark}[theorem]{Remark}

\date{}

\begin{document}

\date{}

\title{Coloring of the square of Kneser graph $K(2k+r,k)$}

\date{December 10, 2014}
\author{
\begin{tabular}{c}
{\sc Seog-Jin KIM\thanks{This paper was supported by Konkuk University in 2014}} \\
[1ex]
{\small
Department of Mathematics Education, Konkuk University, Korea} \\ \vspace{0,3cm}
{\small
{\it E-mail address}: {\tt skim12@konkuk.ac.kr}} \\ 
{\sc Boram PARK\thanks{Corresponding author: borampark@ajou.ac.kr}
} \\
[1ex]
{\small Department of Mathematics, Ajou University, Korea}\\
{\small
{\it E-mail address}: {\tt borampark@ajou.ac.kr}} \\
\end{tabular}
}


\maketitle

\begin{abstract}
The Kneser graph $K(n,k)$ is the graph whose vertices are the $k$-element subsets of an $n$ elements set, with two vertices adjacent if they are disjoint. The square $G^2$ of  a graph $G$ is the graph defined on $V(G)$ such that two vertices $u$ and $v$ are adjacent in $G^2$ if the distance between $u$ and $v$ in $G$ is at most 2.  Determining the chromatic number of the square of the Kneser graph $K(n, k)$ is an interesting graph coloring problem, and is also related with intersecting family problem. The square of $K(2k, k)$ is a perfect matching and the square of $K(n, k)$ is the complete graph when $n \geq 3k-1$.  Hence coloring of the square of $K(2k +1, k)$ has been studied as the first nontrivial case. In this paper, we focus on the question of determining $\chi(K^2(2k+r,k))$ for $r \geq 2$.

Recently, Kim and Park \cite{KP2014}  showed that $\chi(K^2(2k+1,k)) \leq 2k+2$ if $ 2k +1 = 2^t -1$ for some  positive integer $t$.
In this paper, we generalize the result by showing that for any integer $r$ with $1 \leq r \leq k -2$, \\
(a) $\chi(K^2 (2k+r, k)) \leq (2k+r)^r$, if $2k + r = 2^t$  for some integer $t$, and  \\
(b) $\chi(K^2 (2k+r, k)) \leq (2k+r+1)^r$, if $2k + r = 2^t-1$  for some integer $t$.

On the other hand, it was showed in \cite{KP2014} that $\chi(K^2 (2k+r, k)) \leq  (r+2)(3k + \frac{3r+3}{2})^r$ for $2 \leq r \leq k-2$. We improve these  bounds by showing that for  any integer $r$ with $2 \leq r \leq k -2$, we have $\chi(K^2 (2k+r, k)) \leq2 \left(\frac{9}{4}k + \frac{9(r+3)}{8} \right)^r$. Our approach is also related with injective coloring and coloring of Johnson graph.

\end{abstract}


\section{Introduction and Main result}

For a finite set $X$, let ${X \choose k}$ be the set of all $k$-element subsets of $X$.
For $n \geq 2k$, for a finite set $X$ with $n$ elements,
the {\em Kneser graph} $K(n,k)$ is the graph whose vertex set is
${X \choose k}$ and two vertices $A$ and $B$ are adjacent if and only if $A\cap B =\emptyset$.
Kneser graphs have many interesting properties and have been the subject of many researches. The problem of computing the chromatic number of a Kneser graph was conjectured by Kneser  and proved by Lov\'{a}sz~\cite{1978L} that $\chi(K(n,k))=n-2k+2$.

For a simple graph $G$, the {\it square} $G^2$ of $G$ is defined such that $V(G^2) = V(G)$ and two vertices $x$ and $y$ are adjacent in $G^2$ if and only if the distance between $x$ and $y$ in $G$ is at most 2.
We denote  the square of the Kneser graph $K(n,k)$ by $K^2(n,k)$ .
Note that $A$ and $B$ are adjacent in $K^2(n,k)$ if and only if $A\cap B=\emptyset$ or $|A\cap B|\ge 3k-n$.  Therefore, $K^2(n,k)$ is a complete graph
if $n\ge 3k-1$, and  $K^2(n,k)$ is a perfect matching if $n=2k$.
But for  $2k+1\le n \le 3k-2$, the exact value of $\chi(K^2(n,k))$ is not known.  Thus
coloring of the square of $K(2k +1, k)$ has been studied as the first nontrivial case.
Kim and Nakprasit~\cite{2004KN} showed that
 $\chi(K^2(2k+1,k))\le 4k+2$ for $k \geq 2$.
Later, Chen, Lih, and Wu~\cite{2009CLW} improved the bound as $\chi(K^2(2k+1,k))\le 3k+2$ for $k\ge 3$.
Recently, Kim and Park \cite{KP2014} showed that $\chi(K^2 (2k+1, k)) \leq \frac{8}{3}k +\frac{20}{3}$ for $k \geq 2$.

In this paper, we focus on the question of determining $\chi(K^2(2k+r,k))$ for $r \geq 2$.
Note that $K^2 (2k+r, k)$ is a ${{k+r} \choose r}^2$-regular graph.
Hence we have the trivial upper bounds, $\chi(K^2 (2k+r, k)) \leq {{k+r} \choose r}^2$ for $k \geq 3$.
On the other hand, we can show that ${k + 2r \choose r} \leq \omega(K^2 (2k+r, k))$.
In fact, let $S$ be the set of all $k$-subsets of $\{1, 2, \ldots, 2k+r \}$ which contains $\{1, 2, \ldots, k-r \}$, and then  $|A \cap B| \ge k-r$ for all $A, B$ in $S$.
Thus the subgraph of $K^2(2k+r, k)$ induced by $S$
is a clique of size ${{k+2r} \choose r}$.  Hence ${k + 2r \choose r} \leq \chi (K^2 (2k+r, k))$.
Thus  trivial bounds on $\chi(K^2 (2k+r, k))$ are ${k + 2r \choose r} \le \chi(K^2 (2k+r, k)) \leq {{k+r} \choose r}^2$.
Note that finding the exact value of $\omega(K^2 (2k+r, k))$ is related with very difficut intersecting family problems.  Hence it would be an interesting problem to decide the exact value of $\omega(K^2 (2k+r, k))$.

It is a well-known  that
$\big( \frac{m}{t} \big) ^t \leq {m \choose t} \leq \big(\frac{m \cdot e}{t} \big)^t$ for $1 \leq t \leq m$.  Thus for $2 \leq r \leq k-2$,
the trivial bounds on $\chi(K^2 (2k+r, k))$ can be stated as
$ a_r k^r \leq \chi(K^2 (2k+r, k)) \leq b_r k^{2r}$ for some constant $a_r$ and $b_r$.
Recently, Kim and Park \cite{KP2014} showed that
$ c_r k^r \leq \chi(K^2 (2k+r, k)) \leq d_r k^{r}$ for some constants $c_r$ and $d_r$ for $2 \leq r \leq k-2$.
More precisely, they showed that $\chi(K^2 (2k+r, k)) \leq  (r+2)(3k + \frac{3r+3}{2})^r$ for $1 \leq r \leq k-2$.  In this paper, we improve these bounds as follows.

\begin{theorem}\label{main-result-general}
For  any integer $r$ with $1 \leq r \leq k -2$,
\[\chi(K^2 (2k+r, k)) \leq2 \left(\frac{9}{4}k + \frac{9(r+3)}{8} \right)^r.\]
\end{theorem}

For special case, better upper bounds were known.
Kim and Park \cite{KP2014} showed that  $\chi(K^2 (2k+1, k)) \leq 2k+2$ if  $2k + 1 = 2^t -1$  for some integer $t$.
In this paper, we generalize this result as follows.

\begin{theorem}\label{main-result-special}
For any integer $r$ with $1 \leq r \leq k -2$,  we have \\
(a) $\chi(K^2 (2k+r, k)) \leq (2k+r)^r$, if $2k + r = 2^t$  for some integer $t$, \\
(b) $\chi(K^2 (2k+r, k)) \leq (2k+r+1)^r$, if $2k + r = 2^t-1$  for some integer $t$.
\end{theorem}

An \textit{injective $k$-coloring} of a graph is a $k$-coloring so that two vertices sharing a neighbor must have different colors. The \textit{injective chromatic number} $\chi_i(G)$ of a graph $G$ is the smallest integer $k$ such that $G$ admits an injective $k$-coloring. (See \cite{Borodin2013} for a survey).
 The coloring of $G^2$ is related with injective coloring of $G$.  In general, for each graph $G$, we have $\Delta(G) \leq \chi_i (G) \leq \chi(G^2)$.
Our approach also gives an upper bound on the injective coloring of $K^2 (2k+r, k)$.  In Theorem \ref{injective}, we show that
$\chi_i(K(2k+r,k)) \le \left(\frac{9}{4}k + \frac{9(r+3)}{8} \right)^r$ for any integer $r$ with $1 \leq r \leq k -2$.

Johnson graph $J(n, k)$  is the graph whose
vertex set consists of all $k$-subsets of a fixed $n$-set, and two $k$-sets $A$ and $B$ are adjacent if and only if $|A \cap B| = k-1$.  Let $J^m (2k+r, k)$ be the $m$th power of Johnson graph $J(2k+r, k)$, defined by the graph with the same vertex set with two vertices adjacent if they are in distance at most $m$ in $J(2k+r, k)$.
Then $A$ and $B$ are adjacent in $J^m (2k+r, k)$ if and only if $k-m \leq |A \cap B| \leq k-1$.
Thus the coloring of $r$th power of the Johnson graph $J(2k+r, k)$ is the same as the injective coloring of $K(2k+r, k)$.  In general, in Corollary \ref{Johnson} we show that $\chi(J^m (2k+r, k)) \leq
\left(\frac{9}{4}k + \frac{9(r+3)}{8} \right)^m$ for any integer $m$ with $2 \leq m \leq r$.

\medskip
On the other hand, Kim and Park \cite{KP2014} showed the for any real number $\epsilon > 0$,
\[\limsup_{k \rightarrow \infty} \frac{\chi(K^2(2k+1,k))}{k} \le  2 +\epsilon.\]
In this paper, we give a similar result for $r \geq 2$, that is,
for any real number $\epsilon > 0$, and for any integer $r \geq 2$,
\begin{equation*}\label{infinity-r}
\limsup_{k \rightarrow \infty} \frac{\chi(K^2(2k+r,k))}{k^r} \le  2 \cdot 2^r+\epsilon.
\end{equation*}

\medskip

The rest of the paper is organized as follows.
In Sections \ref{section-injective} and \ref{section-special-case}, we will give proofs of Theorems \ref{main-result-general} and \ref{main-result-special}, respectively.
In Section 4, we will discuss  asymptotic results.


\section{Proof of Theorem \ref{main-result-general}} \label{section-injective}

Throughout this section, fix positive integers $k$ and $r$ such that $1\le r\le k-2$.
Let $F$ be a finite field, and let  $0_F$ and $1_F$ be the identity under addition and  multiplication, respectively, in $F$.
For any nonnegative integer $i$ and for $Z\subset F$, let $e_i(Z)$ be the sum of all distinct products of $i$ distinct elements in $Z$.
That is, when $Z=\{ z_1,\ldots,z_n\}$,
\[ {\displaystyle e_i(Z)=\sum_{1 \leq {t_1}<{t_2}<\cdots<t_i \leq n} z_{t_1}z_{t_2}\cdots z_{t_i} }. \]
Define $e_0(Z)=1_F$ and $e_i(Z)=0_F$ if $i>|Z|$.
For example, when $Z = \{x_1, x_2, x_3, x_4 \}$,  $e_1 (Z) = x_1 + x_2 + x_3 + x_4$, $e_2(Z)=x_1x_2+x_1x_3+x_1x_4+x_2x_3+x_2x_4+x_3x_4$,
and $e_5(Z)=0$.

\medskip

The following is an easy observation, but it is useful.

\begin{lemma}\label{lem_basic}
For two disjoint sets $X$ and $Y$,
\[ e_i(X\cup Y) = \sum_{s=0}^{i} e_s(X)e_{i-s}(Y).\]
\end{lemma}

Let ${F \choose k}$ denote the set of $k$-subsets of $F$.
Define the function $f:{F \choose k} \rightarrow F^r$ by
\begin{equation}\label{main-function}  f(A)=(e_1(A), e_2(A), \ldots, e_r(A)) . \end{equation}

\begin{lemma} \label{lem:intersect}
Let $F$ be a  finite field.  If $A$ and $B$ are subsets of $F$ such that $|A|=|B|=k$ and $k-r\le |A\cap B| \le k-1$, then $f(A)\neq f(B)$ where $f$ is the function defined in (\ref{main-function}).
\end{lemma}

\begin{proof}
We will prove the lemma by contradiction.
Suppose that $f(A) = f(B)$.  It implies that
$e_i (A) = e_i(B)$ for all $i\in \{1,2,\ldots,r\}$.
Note that $|A\cap B|=k-q$ for some $q\in \{1,2,\ldots,r\}$.
Let $A'= A\setminus B =\{a_1,a_2,\ldots,a_q\}$ and $B' = B\setminus A =\{b_1,b_2,\ldots,b_q\}$.

We will show that  $e_i (A') = e_i(B')$ for all $i\in \{1,2,\ldots,q\}$ by the induction on $i$.
Since $A'$ and $A\cap B$ are disjoint and $B'$ and $A\cap B$ are disjoint, we have
\[ e_1 (A) = e_1 (A') + e_1 (A \cap B) = e_1 (B') + e_1 (A \cap B) = e_1 (B). \]
Thus $e_1(A')=e_1(B')$. Therefore the basis step holds.

Now for some $i\in \{1,2,\ldots,q\}$, assume that $e_s(A') = e_s(B')$ for all $s \le i$.
We will show that  $e_{i+1}(A') = e_{i+1}(B')$.
On the other hand, by Lemma~\ref{lem_basic}, we have
\begin{eqnarray*}
&&e_{i+1}(A)=e_{i+1}(A')+\sum_{s=0}^{i} e_s(A')e_{i+1-s}(A\cap B), \\
&&e_{i+1}(B)=e_{i+1}(B')+\sum_{s=0}^{i} e_s(B')e_{i+1-s}(A\cap B).
\end{eqnarray*}
Note that $i +1 \leq r$.
Thus from the condition $e_{i+1}(A)=e_{i+1}(B)$,
\begin{equation} \label{e(i+1)}
e_{i+1}(A')+\sum_{s=0}^{i} e_s(A')e_{i+1-s}(A\cap B) = e_{i+1}(B')+
\sum_{s=0}^{i} e_s(B')e_{i+1-s}(A\cap B).
\end{equation}
Since $e_{s}(A')=e_{s}(B')$ for all $0 \leq s \leq i$ by the induction hypothesis, from (\ref{e(i+1)}) we have
\[ e_{i+1}(A')=e_{i+1}(B').\]
Thus by the induction, we have $e_i (A') = e_i(B')$ for all  $i\in \{1,2,\ldots,q\}$.

\medskip
Note that
\begin{eqnarray*}
(b_1-a_1)(b_1-a_2)\cdots (b_1-a_q)
&=&b_1^q - e_1(A')b_1^{q-1}+ \cdots+ e_q(A') \\
&=&b_1^q - e_1(B')b_1^{q-1}+ \cdots+ e_q(B')\\
&=&(b_1-b_1)(b_2-b_1)\cdots (b_q-b_1) =0_F.
\end{eqnarray*}
Hence $(b_1-a_1)(b_1-a_2)\cdots (b_1-a_q)= 0_F$.
Thus there is an element in $a_i\in A'$ such that $b_1 - a_i = 0_F$, since $F$ is a field.
Then $b_1 = a_i$, which is a contradiction to the fact that $A'$ and $B'$ are disjoint.
\end{proof}

The following is a recent result on the existence of a prime number in a certain interval.

\begin{theorem} \label{prime-2n} {\cite{prime2}}
For any  integer $ n \geq 2$, there  is a prime number $p$ such that $n \le  p \le \frac{9(n+3)}{8}$.
\end{theorem}

Note that the function $f$ in (\ref{main-function}) admits an injective coloring of the Kneser graph $K(2k+r, k)$ by Lemma \ref{lem:intersect}.
Thus, from Lemma \ref{lem:intersect} and Theorem \ref{prime-2n}, we have the following result.

\begin{theorem}\label{injective}
For any integer $1 \leq r \leq k-2$, we have $\chi_i(K(2k+r,k)) \le \left(\frac{9}{4}k + \frac{9(r+3)}{8} \right)^r$.
\end{theorem}

\begin{proof}
Let $k$ and $r$ be positive integers with $1 \leq r \leq k-2$.
From Theorem \ref{prime-2n}, there is a prime $p$ such that $ 2k+r \le  p \le \frac{9(2k+r+3)}{8} =  \frac{9}{4}k + \frac{9(r+3)}{8}$.
Let $F = \mathbb{Z}_p$ be a finite field with $ 2k+r \le |F| \le \frac{9}{4}k + \frac{9(r+3)}{8}$.

Now we take $X\subset F$ such that $|X|=2k+r$.
We define the Kneser graph $K(2k+r, k)$ on the ground set $X$.
Note that two vertices $A$ and $B$ have  a common neighbor in
$K(2k+r, k)$ if and only if $k-r\le |A\cap B| \le k-1$.  And note that $f(A)\neq f(B)$ when $k-r\le |A\cap B|  \le k-1$ by Lemma~\ref{lem:intersect}.
Thus $f$ gives  an injective coloring of $G$.  This implies that
$\chi_i(K(2k+r, k))\le |F|^r \leq \left(\frac{9}{4}k + \frac{9(r+3)}{8} \right)^r$.
\end{proof}

\medskip
Now,  we will complete the proof of Theorem \ref{main-result-general}.

\bigskip

\noindent{\it Proof of Theorem \ref{main-result-general}.}  Kim and Oum \cite{KO2009} showed that  $\chi(G^2)\le 2\chi_i(G)$ for any graph $G$.
Thus
\[
\chi(K^2 (2k+r, k)) \leq 2 \cdot \chi_i (K(2k+r, k)) = 2  \left(\frac{9}{4}k + \frac{9(r+3)}{8} \right)^r.
\]
This completes the proof of Theorem \ref{main-result-general}.
\qed

\bigskip

Then  from Lemma \ref{lem:intersect} and Theorem \ref{injective},  we have the following corollary about Johnson graph.

\begin{corollary}\label{Johnson}
For $2 \leq m \leq r$,  we have $\chi(J^m (2k+r, k)) \leq
\left(\frac{9}{4}k + \frac{9(r+3)}{8} \right)^m$.
\end{corollary}
\begin{proof}
Let $F$ be the finite field with $|F| = p$ for some prime $p$ such that $ 2k+r \le p \le \frac{9}{4}k + \frac{9(r+3)}{8}$, and $X\subset F$ such that $|X|=2k+r$.
We define the Johnson graph $J(2k+r, k)$ on the ground set $X$.
Then two vertices $A$ and $B$ in $J^m(2k+r,k)$ are adjacent if and only if
$|A\cap B|=k-q$ for some $q\in \{1,2,\ldots,m\}$.
Then the function $f:{F \choose k} \rightarrow F^m$ defined by
\[f(A)=(e_1(A),e_2(A),\ldots, e_m(A)),\]
gives a proper coloring of $\chi(J^m (2k+r, k))$ by the same argument of Lemma \ref{lem:intersect}.
Thus
\[\chi(J^m (2k+r, k)) \leq |F|^m \leq
\left(\frac{9}{4}k + \frac{9(r+3)}{8} \right)^m.
\]
\end{proof}


\section{Proof of Theorem \ref{main-result-special}}  \label{section-special-case}

In Theorem \ref{injective}, we showed that
the function $f$ defined in (\ref{main-function}) gives  an injective coloring of $K(2k + r, k)$.
In this section, we will show that when $2k+r = 2^t$ or $2^t -1$ for some integer $t$,   the function $f$  also gives a proper  coloring of the square of the Kneser graph $K(2k+r,k)$.

\begin{lemma}\label{lem:disjoint}
Let $r \geq 2$.  Let $F$ be a finite field of characteristic 2 and let $X$ be a subset of $F$ with $|X| = 2k+r$ such that $e_{2m+1}(X)=0_F$ for any nonnegative integer $m$ with $2m+1\le r$.
If $A$ and $B$ are two disjoint subsets in $X$ with $|A| = |B| = k$, then
$f(A)\neq(B)$ where  $f$ is the function defined in (\ref{main-function}).
\end{lemma}

\begin{proof}
Suppose that $f(A)=f(B)$.
Then by the definition of $f$,
\begin{equation}\label{eq_basic}
e_i (A) = e_i(B) \mbox{ for all } 1 \leq i \leq r.
\end{equation}

\begin{claim}\label{claim:disjoint}
For any nonnegative integer $m$ with $2m+1\le r$, we have
$e_{2m + 1} ( X\setminus(A\cup B) ) = 0_F$.
\end{claim}

\begin{proof}
Let $S=X\setminus(A\cup B)$.
We will prove the claim by the induction on $m$.
Since $e_{1}(X)=0_F$, we have
\[ 0_F=e_{1}(X)=e_1(S)+e_1(A)+e_1(B). \]
By (\ref{eq_basic}),
\[0_F=e_1(X)=e_1(S)+2e_1(A) = e_1(S).\]
Therefore $e_1(S)=0_F$, so the claim is true when $m=0$.

Now suppose that $e_{2i +1} (S) = 0_F$ for all $i$ such that $1 \leq 2i+1 <2m+1$.
Next, we will show that $e_{2m+1} (S) = 0_F$.  Note that
\begin{equation*} \label{equation-S}
0_F=e_{2m+1}(X)= e_{2m+1}(S) + \sum_{s=1}^{2m+1} e_{s}(A\cup B)e_{2m+1-s}(S).
\end{equation*}
By the induction hypothesis that $e_{2i +1} (S) = 0_F$ for all $i$ with $1 \leq 2i+1 <2m+1$,
it is simplified into
\begin{eqnarray} \label{equation-simple}
0_F=e_{2m+1}(S) +   \sum_{s=1}^{m+1}  e_{2s-1} (A\cup B) e_{2m+1-(2s-1)}(S).
\end{eqnarray}
Now, for any positive integer $s$ such that $2s-1\le r$, we have
\begin{eqnarray*}
e_{2s-1}(A\cup B) &=& \sum_{i=0}^{2s-1} e_{i}(A)e_{2s-1-i}(B)  \\
&=& \sum_{i=0}^{2s-1} e_{i}(A) e_{2s-1-i}(A)\\
& =& \sum_{i=0}^{s-1} \left( e_{i}(A) e_{2s-1-i}(A) + e_{2s-1-i}(A) e_{i}(A)  \right) \\
&=&\sum_{i=0}^{s-1} 2 e_{i}(A) e_{2s-1-i}(A) = 0_F,
\end{eqnarray*}
where the second equality is from (\ref{eq_basic}), and the last equality is from the fact that the characteristic of $F$ is 2.  Thus from (\ref{equation-simple}), we have $e_{2m+1}(S)=0_F$.  This completes  the proof of Claim \ref{claim:disjoint}.
\end{proof}

Now for $S=X\setminus(A\cup B)$, we consider a polynomial $g(x)=\Pi_{a\in S}(x-a)$ in $F[x]$. 
Note that $|S|=r$, since $|X| = 2k+r$ and  $A \cap B = \emptyset$.
Then
\begin{equation*} \label{function-g(x)}
g(x)=x^r-e_1(S)x^{r-1}+e_2(S)x^{r-2}+\cdots + (-1)^i e_i(S) x^{r-i} + \cdots + (-1)^r e_r(S).
\end{equation*}
Note that $S$ is the set of all the zeros of the polynomial $g(x)$.
By Claim~\ref{claim:disjoint},
$e_{i}(S)=0_F$ for any positive odd integer $i$ with $i\le r$. Thus
if $r$ is even, then
\[g(x)=x^r+e_2(S)x^{r-2}+e_4(S)x^{r-4}+\cdots + e_r(S). \]
On the other hand, if $r$ is odd, then
\[g(x)=x \big(x^{r-1}+e_2(S)x^{r-3}+e_4(S)x^{r-5}+\cdots + e_{r-1}(S) \big).\]

\noindent {\bf Case 1}: when $r$ is even

Define
\[h(y)=y^{\frac{r}{2}} + e_2(S) y^{\frac{r-2}{2}} + e_4(S) y^{\frac{r-4}{2}} + \cdots + e_r(S).\]
Then by substituting $x^{2}=y$, we have $h(y)=g(x)$.
Note that for any $a \in S$, $b = a^2$ is a root of $h(y)$, but the polynomial $h(y)$ has at most $\frac{r}{2}$ roots. Since $r\ge 2$, it holds that $|S| = r > \frac{r}{2}$.
Thus there are distinct elements $p$ and $q$ in $S$ such that $p^2=q^2$ .
Then $p^2+q^2=0_F$ since the characteristic of $F$ is 2.
Since $(p+q)^2=p^2+q^2$, $(p+q)^2=0_F$. Therefore $p+q=0_F$. It implies that $p = q$, which is a contradiction to the fact that $p$ and $q$ are distinct.

\medskip
\noindent {\bf Case 2}: when $r$ is odd

Then $g(x)=x\phi(x^2)$ where
\[\phi(y)=y^{\frac{r-1}{2}} + e_2(S) y^{\frac{r-3}{2}} + e_4(S) y^{\frac{r-5}{2}} + \cdots + e_{r-1}(S).\]
Note that for any $a \in S$, $b = a^2$ is a root of $\phi(y)$, but
$g(x)$ has at most $ \frac{r-1}{2} +1$ roots.
Since $r\ge2$, it holds that $|S| = r > \frac{r-1}{2} +1$.
It follows that $p^2=q^2$ for some two distinct elements $p,q\in S$.
Then $p^2+q^2=0_F$, which is a contradiction to the fact that $p$ and $q$ are distinct.
This completes the proof of Lemma \ref{lem:disjoint}. \end{proof}

\medskip

 The following is the main result of this section.

\begin{theorem}\label{main-result0}
Let $F$ be a finite field of characteristic 2.
For $2\le r\le k-1$, if  there is a subset $X$ of $F$ such that $|X|=2k+r$ and $e_{2m+1}(X)=0_F$ for all nonnegative integer $m$ with $2m+1\le r$, then
$\chi(K^2 (2k+r, k)) \leq |F|^r$.
\end{theorem}

\begin{proof}
Let $F$ be a finite field such that $|F|=2^t$, and
$X$ be a subset of $F$ such that $|X|=2k+r$ and $e_{2m+1}(X)=0_F$ for all integers $m$ such that $2m+1\le r$.
We define the Kneser graph $K(2k+r, k)$ on the ground set $X$, and denote  $G = K(2k+r, k)$.
Let $f$ be the function defined in (\ref{main-function}).
Let $AB$ be an edge in $G^2$ for some two vertices $A$, $B$ of $G^2$.
If $A \cap B \neq \emptyset$,  then $k-r\le |A\cap B| < k-1$, and so  $f(A)\neq f(B)$ by Lemma~\ref{lem:intersect}.
If $A \cap B = \emptyset$,  then  $f(A)\neq f(B)$ by Lemma~\ref{lem:disjoint}.
Hence $f$ gives a proper coloring of $G^2$, and $\chi(G^2)$ is at most the size of the range of $f$.
\end{proof}

\bigskip
Now we are ready to prove Theorem \ref{main-result-special}.

\medskip

\noindent{\it Proof of Theorem \ref{main-result-special}.}  The case where $r=1$ is shown in \cite{KP2014}. We suppose that $r\ge 2$.
Let $F$ be a field with $|F|=2^t$.
Since $F$ is the splitting field of the polynomial $x^{2^t}-x$ in $\mathbb{Z}_2[x]$,
$e_{2m+1}(F)$ is the coefficient of  the term $x^{2^t-(2m+1)}$ in the polynomial $x^{2^t}-x$.
Note that if $2m+1\le r$, then $2m+1\le r=2^t- 2k  \le 2^t-3$ for $k \geq 2$.
Thus
\begin{eqnarray}\label{eq_F}
&& e_{2m+1}(F)=0_F \quad \text{for any nonnegative integer }m\text{ such that } 2m+1\le r.
\end{eqnarray}
Now, put $X = F$, then  (a) holds by Theorem~\ref{main-result0}.

On the other hand, from (\ref{eq_F}), we have
\begin{eqnarray*}\label{eq_F-two}
&& e_{2m+1}(F \setminus \{0_F\})=0_F \quad \text{for any nonnegative integer }m\text{ such that } 2m+1\le r.
\end{eqnarray*}
Put $X =  F \setminus \{0_F\}$, then (b) holds by Theorem~\ref{main-result0}.  This completes the proof of Theorem \ref{main-result-special}.
\qed

\bigskip
Furthermore, we have the following corollary.

\begin{corollary} \label{cor:main-result1}
Let $k$ and $r$ be integers such that $k\ge 2$ and $1 \leq r \leq k -2$.
Then the following holds.
\begin{itemize}
\item[(i)] if $2k + r = 2^t -2^{t'}$ for some integers $t\ge t'\ge \log_2(r+2)$,
then $\chi(K^2 (2k+r, k)) \leq (2k+r +2^{t'})^r$.
\item[(ii)] if $2k + r = 2^t -2^{t'}+1$ for some integers $t\ge t'\ge \log_2(r+2)$,
then $\chi(K^2 (2k+r, k)) \leq (2k+r +2^{t'})^r$.
\end{itemize}
\end{corollary}

\begin{proof}
Let $F$ be a field with $|F|=2^t$.

\begin{claim}\label{claim:subfield}
Let $X$ and $Y$ be subsets of $F$ such that
$e_{2m+1}(X)=e_{2m+1}(Y)=0_F$ for all integers $m$ satisfying $1 \leq 2m+1\leq r$.
If $X \subset Y$ then for any integer $m$ satisfying $1 \leq 2m+1 \leq r$,
we have $e_{2m+1}(Y \setminus X)=0_F$.
\end{claim}

\begin{proof}
We will prove by the induction on $m$.
By Lemma~\ref{lem_basic}, by the assumption that $e_1(X)=e_1(Y)=0_F$, we have
\[ 0_F=e_{1}(Y) = e_1(X)+e_{1}(Y\setminus X)=e_{1}(Y\setminus X),\]
and so the basis step holds.

Suppose that $e_{2q+1}(Y\setminus X)=0_F$ for all
integers $q$ such that $0 \leq q < m$ for some $m\ge1$.
By Lemma~\ref{lem_basic}, 
we have
\begin{eqnarray*}
0_F=e_{2m+1}(Y) &=& \sum_{i=0}^{2m+1} e_{i}(X) e_{2m+1-i}(Y \setminus X)  \\
&=& \sum_{i=0}^{m} e_{2i}(X) e_{2m+1-2i}(Y \setminus X)\\
&=& e_{2m+1}(Y\setminus X).
\end{eqnarray*}
Note that the third equality is from the assumption that $e_{2q+1}(X)=0_F$ for all $2q+1\le r$, and
the last equality is from the induction hypothesis.
This completes the proof of Claim \ref{claim:subfield}.
\end{proof}

Since $F$ is a splitting field of the polynomial $x^{2^t}-x$ in $\mathbb{Z}_2[x]$,
$e_{2m+1}(F)=0_F$ for any nonnegative integer $m$ such that $2m+1\le r$ as it
is the coefficient of  the term $x^{2^t-(2m+1)}$ in the polynomial $x^{2^t}-x$.
Let $t$ and $t'$ be integers with $t\ge t'\ge \log_2(r+2)$.
Let $F'$ be a subfield of $F$ with  $|F'|=2^{t'}$.
Since $F'$ is the splitting field of $x^{2^{t'}}-x$ and $2^{t'}-2 \ge r$, we have
\begin{eqnarray}\label{eq_F'}
&& e_{2m+1}(F')=0_F\quad \text{for any nonnegative integer }m\text{ such that } 2m+1\le r.
\end{eqnarray}
When $2k+r=2^t-2^{t'}$, put $X=F\setminus F'$.  Then  $|X|=2k+r$ and $e_{2m+1}(X)=0_F$ by Claim~\ref{claim:subfield}, (\ref{eq_F}), and (\ref{eq_F'}).  Thus  (i) holds by Theorem~\ref{main-result0}.

When $2k+r=2^t-2^{t'}+1$, put $X=(F\setminus F')\cup\{0_F\}$.  Then $|X|=2k+r$ and
$e_{2m+1}(X)=0_F$ by
Claim~\ref{claim:subfield},  (\ref{eq_F}) and~(\ref{eq_F'}). Thus (ii) holds by Theorem~\ref{main-result0}.
\end{proof}

\section{Further Discussion}

Kim and Park \cite{KP2014} show that
for integers $k$ and $n$,
if $2k+1=(2^n-1)p+r$ where  $p\ge 1$, $n\ge2$, and  $0\le r\le 2^n-2$,
\begin{equation} \label{upper-K(2k+1,k)}
\chi(K^2(2k+1,k))\leq \frac{2^{n+1}}{2^n-1}k + \frac{2^{n}(2^{n+1}-3)}{2^n-1}.
\end{equation}

From (\ref{upper-K(2k+1,k)}), the following corollary was obtained.

\begin{corollary} [{\cite{KP2014}}, Corollary 2.8] \label{large k}
For any fixed real number $\epsilon > 0$, there exists a positive integer $k_0$ depending on $\epsilon$ such that
\[
\chi(K^2 (2k+1, k)) \leq (2 + \epsilon)(k +  \sqrt{2k+1}),
\]
for any positive integer $k \geq k_0$.  Thus for any fixed real number $\epsilon > 0$,
\[
\limsup_{k \rightarrow \infty} \frac{\chi(K^2 (2k+1, k))}{k} \leq 2 + \epsilon.
\]
\end{corollary}

From the previous results of $\chi(K^2(2k+1,k))$ in \cite{2009CLW, 2004KN, KP2014},
we have  $\chi(K^2(2k+1,k))\leq \alpha k + \beta$ for some constant real numbers $\alpha$ and $\beta$.  We would like to find the optimal value of $\alpha$ for some constant $\beta$.
The equality (\ref{upper-K(2k+1,k)}) and Corollary~\ref{large k} imply that the leading coefficient $\alpha$ in $\alpha k + \beta$ can be decreased as  we increase $k$.
It was showed in \cite{KP2014} that  $\chi(K^2(2k+1,k))\leq \frac{8}{3}k + \frac{20}{3}$ for $k \geq 3$.
However, we have understood that  better  upper bounds on $\chi(K^2 (2k+1, k))$ can be obtained
by increasing $k$ slightly.  If $k \geq 7$, we can assume that $n \geq 4$ and $p = 1$ in $2k+1=(2^n-1)p+r$.
Thus we have the following better upper bounds.

\begin{corollary}  \label{better-uppper}
For $k \geq 7$, we have
\[ \chi(K^2(2k+1,k))\leq \frac{32}{15}k + 32. \]
\end{corollary}

\medskip

Now we discuss a similar asymptotic result for $\chi(K^2(2k+r,k))$ where $2\le r \le k-1$.
We start from the following theorem in \cite{prime3}.

\begin{theorem}\label{prime3} {\cite{prime3}}
For any integer $n\ge 3275$, there is a prime $p$ such that $n \le p \le \left( 1+\frac{1}{2 \ln^2 n}\right)n$.
\end{theorem}

\begin{theorem}\label{injective2}
For any positive real number $\epsilon$,
\[\limsup_{k \rightarrow \infty} \frac{\chi_i(K(2k+r,k))}{k^r} \le  2^r+\epsilon.\]
\end{theorem}

\begin{proof}
For any fixed real number $\epsilon > 0$,
let \[\delta=\frac{(2^{r}+\epsilon)^{1/r}-2}{2}.\]
Then clearly, $\delta>0$, and so
there exists an integer $n_0$ such that $\frac{1}{2 (\ln n_0)^2}  < \delta$.
Let $n_0$ be a fixed integer with $n_0\ge 3275$ and satisfying $\frac{1}{2 (\ln n_0)^2}  < \delta$.  Then for any integer $n$ with $n\ge n_0$, we have
$\frac{1}{2 (\ln n)^2} <\delta$.
By Theorem~\ref{prime3}, for any integer $n$ with $n\ge n_0$,
there exists a prime $p_n$ such that
\begin{equation} \label{prime}
n\le p_n \le \left( 1+\frac{1}{2 (\ln n)^2}\right)n < n+\delta n .
\end{equation}
Let $k$ be a sufficiently large integer such that $2k+r \ge n_0$.
Then for $n = 2k+r$, there exists a prime $p_{2k+r}$ such that $p_{2k+r} <(2k+r)+ \delta(2k+r)$ from (\ref{prime}). Let $F$ be a finite field with
$|F|=p_{2k+r}$.
Let $G$ be the Kneser graph $K(2k+r, k)$ on the ground set $X$ where $X$ is a subset of  $F$.
Then
\[\chi_i(G) \le |F|^r=p_{2k+r}^r <  \big((2k+r)+ \delta(2k+1) \big)^r= \big(  (2+2\delta)k + r(1+\delta) \big)^r.\]
Since
\[2+2\delta=2+(2^{r}+\epsilon)^{1/r}-2=(2^{r}+\epsilon)^{1/r},\]
we have
\[\limsup_{k \rightarrow \infty} \frac{\chi_i(K(2k+r,k))}{k^r} \le 2^r+\epsilon.\]
\end{proof}


From the fact that $\chi(G^2) \leq 2 \chi_i (G)$ for any graph $G$, by the same argument in
Theorem \ref{injective2},
we have the following corollary.

\begin{corollary}\label{infinity-r}
For any real number $\epsilon > 0$, and for any integer $r \geq 2$, we have
\[\limsup_{k \rightarrow \infty} \frac{\chi(K^2(2k+r,k))}{k^r} \le  2 \cdot 2^{r}+\epsilon.\]
\end{corollary}

\begin{remark} \rm
Note that the result $\chi(K^2 (2k+r, k)) \leq  (r+2)(3k + \frac{3r+3}{2})^r$ in \cite{KP2014} implies that
\[ \limsup_{k \rightarrow \infty} \frac{\chi(K^2(2k+r,k))}{k^r} \leq (r+2) \cdot 3^r + \epsilon. \]
Hence Corollary \ref{infinity-r} implies that the leading term of the upper bound on $\limsup_{k \rightarrow \infty} \frac{\chi(K^2(2k+r,k))}{k^r}$ is reduced  from $(r+2) \cdot 3^r$ to $2 \cdot 2^r$.
\end{remark}





\begin{thebibliography}{00}
\bibitem{Borodin2013}
Borodin, O.V.: Colorings of plane graphs: A survey, {\it Discrete Math.} {\bf 313} (2013), 517–-539.

\bibitem{2009CLW}
Chen, J.-Y.,  Lih, K.-W., Wu, J.:
Coloring the square of the Kneser graph $KG(2k+1,k)$ and the Schrijver
graph $SG(2k+2,2)$,
{\it Discrete Appl. Math.} {\bf 157} (2009), 170--176.


\bibitem{De_Er_PF}  Deza, M., Erd\H{o}s, P., Frankl, P.:
 Intersection properties of the systems of finite sets,
{\it Proc. London Math. Soc.} {\bf 36} (1978), 369--384.



\bibitem{prime2} Paz, G.A.: On the interval [n; 2n]: Primes, composites and perfect powers,
{\it Gen. Math. Notes} {\bf 15} (2013), 1--15.


\bibitem{prime3} Dusart, P.: Autour de la fonction qui compte le nombre de nombres pre-
miers, Ph.D. Thesis, Universit$\acute{\text{e}}$ de Limoges, (1998).


\bibitem{2004KN}
Kim, S.-J.,  Nakprasit, K.:
On the chromatic number of the square of the Kneser graph $K(2k+1,k)$,
{\it  Graphs Combin.} {\bf 20} (2004), 79--90.


\bibitem{KO2009}
Kim, S.-J., Oum, S.: Injective chromatic number and chromatic number of the square of graphs, manuscript, 2009.

\bibitem{KP2014}
Kim, S.-J., Park, B.:
Improved bounds on the chromatic numbers of the square of Kneser graphs,
{\it Discrete Math.} \textbf{315} (2014), 69–-74.

\bibitem{1978L}
Lov\'{a}sz, L.:
Kneser's conjecture, chromatic number and homotopy,
{\it J. Comb. Theory, Ser. A} {\bf 25} (1978), 319--324.

\bibitem{Newton2003}
Min\'{a}\v{c}, J.:
Newton's identities once again!,
{\it Amer. Math. Monthly} {\bf 110} (2003),  232--234.

\end{thebibliography}
\end{document}